\documentclass[10pt,reqno]{amsart}

\usepackage{amsthm, mathrsfs, amsmath, amstext, amsxtra, amsfonts, dsfont, amssymb}
\usepackage{color}
\usepackage{lmodern}
\definecolor{green_dark}{rgb}{0,0.6,0}
\usepackage[colorlinks, linkcolor=red, citecolor=blue, urlcolor=green_dark, pagebackref, hypertexnames=false]{hyperref}

\newcommand{\R}{\mathbb R}
\newcommand{\C}{\mathbb C}

\newcommand{\ep}{\epsilon}
\newcommand{\Gact} {\gamma_{\text{c}}}

\newcommand{\re}[1]{\mbox{Re} #1} 
\newcommand{\im}[1]{\mbox{Im} #1} 
\newcommand{\reemph}[1]{\mbox{\emph{Re}} #1} 
\newcommand{\imemph}[1]{\mbox{\emph{Im}} #1}

\newcommand{\defendproof}{\hfill $\Box$} 

\newtheorem{theorem}{Theorem}[section]
\newtheorem{lem}[theorem]{Lemma} 

\newtheorem{coro}[theorem]{Corollary} 
\theoremstyle{definition}

\newtheorem{rem}[theorem]{Remark}

\title[Blowup for inhomogeneous NLS]{Blowup of $H^1$ solutions for a class of the focusing inhomogeneous nonlinear Schr\"odinger equation} 

\author[V. D. Dinh]{Van Duong Dinh}
\address[V. D. Dinh]{Institut de Math\'ematiques de Toulouse UMR5219, Universit\'e Toulouse CNRS, 31062 Toulouse Cedex 9, France}
\email{dinhvan.duong@math.univ-toulouse.fr}

\keywords{Inhomogeneous nonlinear Schr\"odinger equation; Blowup; Virial identity; Gagliardo-Nirenberg inequality}
\subjclass[2010]{35B44, 35Q55}

\begin{document}

\begin{abstract}
We consider a class of focusing inhomogeneous nonlinear Schr\"odinger equation
\[
i\partial_t u + \Delta u + |x|^{-b} |u|^\alpha u = 0, \quad u(0)=u_0 \in H^1(\R^d),
\]
with $0<b<\min\{2,d\}$ and $\alpha_\star\leq \alpha <\alpha^\star$ where $\alpha_\star =\frac{4-2b}{d}$ and $\alpha^\star=\frac{4-2b}{d-2}$ if $d\geq 3$ and $\alpha^\star = \infty$ if $d=1,2$. In the mass-critical case $\alpha=\alpha_\star$, we prove that if $u_0$ has negative energy and satisfies either $xu_0 \in L^2$ or $u_0$ is radial with $d\geq 2$, then the corresponding solution blows up in finite time. Moreover, when $d=1$, we prove that if the initial data (not necessarily radial) has negative energy, then the corresponding solution blows up in finite time. In the mass and energy intercritical case $\alpha_\star< \alpha <\alpha^\star$, we prove the finite time blowup for radial negative energy initial data as well as the finite time blowup below ground state for radial initial data in dimensions $d\geq 2$. This result extends the one of Farah in \cite{Farah} where the blowup below ground state was proved for data in the virial space $H^1\cap L^2(|x|^2 dx)$ with $d\geq 1$.
\end{abstract}

\maketitle

\section{Introduction}
\setcounter{equation}{0}
In this paper, we consider the Cauchy problem for the inhomogeneous nonlinear Schr\"odinger equation
\begin{equation}
\left\{
\begin{array}{ccl}
i\partial_t u + \Delta u + \mu |x|^{-b} |u|^\alpha u &=& 0, \\
u(0)&=& u_0,
\end{array} 
\right. \tag{INLS}
\label{INLS}
\end{equation}
where $u: \R \times \R^d \rightarrow \C$, $u_0:\R^d \rightarrow \C$, $\mu = \pm 1$ and $\alpha, b>0$. The parameters $\mu=1$ (resp. $\mu=-1$) corresponds to the focusing (resp. defocusing) case. The case $b=0$ is the well-known nonlinear Schr\"odinger equation which has been studied extensively over the last three decades. The inhomogeneous nonlinear Schr\"odinger equation arises naturally in nonlinear optics for the propagation of laser beams, and it is of a form
\begin{align}
i \partial_t u +\Delta u + K(x) |u|^\alpha u =0. \label{general INLS}
\end{align}
The $\eqref{INLS}$ is a particular case of $(\ref{general INLS})$ with $K(x)=|x|^{-b}$. The equation $(\ref{general INLS})$ has been attracted a lot of interest in a past several years. Berg\'e in \cite{Berge} studied formally the stability condition for soliton solutions of $(\ref{general INLS})$. Towers-Malomed in \cite{TowersMalomed} observed by means of variational approximation and direct simulations that a certain type of time-dependent nonlinear medium gives rise to completely stabe beams. Merle in \cite{Merle} and Rapha\"el-Szeftel in \cite{RaphaelSzeftel} studied $(\ref{general INLS})$ for $k_1<K(x)<k_2$ with $k_1, k_2>0$. Fibich-Wang in \cite{FibichWang} investigated $(\ref{general INLS})$ with $K(x):=K(\ep|x|)$  where $\ep>0$ is small and $K \in C^4(\R^d) \cap L^\infty(\R^d)$. The case $K(x) = |x|^{b}$ with $b>0$ is studied by many authors (see e.g. \cite{Chen, LiuWangWang, Zhu} and references therein). \newline
\indent In order to recall known results for the $\eqref{INLS}$, let us give some facts for this equation.  We first note that the $\eqref{INLS}$ is invariant under the scaling,
\[
u_\lambda(t,x):= \lambda^{\frac{2-b}{\alpha}} u(\lambda^2 t, \lambda x), \quad \lambda>0.
\]
An easy computation shows
\[
\|u_\lambda(0)\|_{\dot{H}^\gamma} = \lambda^{\gamma+\frac{2-b}{\alpha}-\frac{d}{2}} \|u_0\|_{\dot{H}^\gamma}.
\]
Thus, the critical Sobolev exponent is given by 
\begin{align}
\Gact := \frac{d}{2}-\frac{2-b}{\alpha}. \label{critical exponent}
\end{align}
Moreover, the $\eqref{INLS}$ has the following conserved quantities:
\begin{align}
M(u(t))&:= \int |u(t,x)|^2 dx = M(u_0), \label{mass conservation} \\
E(u(t)) &:= \int \frac{1}{2}|\nabla u(t,x)|^2 - \frac{\mu}{\alpha+2} |x|^{-b}|u(t,x)|^{\alpha+2} dx = E(u_0). \label{energy conservation} 
\end{align}
\indent The well-posedness for the $\eqref{INLS}$ was first studied by Genoud-Stuart in \cite[Appendix]{GenoudStuart} (see also \cite{Genoud-1D}) by using the argument of Cazenave \cite{Cazenave}. Note that the existence of $H^1$ solutions to $\eqref{INLS}$ is shown by the energy method which does not use Strichartz estimates. More precisely, Genoud-Stuart showed that the focusing $\eqref{INLS}$ with $0<b<\min\{2,d\}$ is well posed in $H^1$: 
\begin{itemize}
\item locally if $0<\alpha<\alpha^\star$,
\item globally for any initial data if $0<\alpha <\alpha_\star$,
\item globally for small initial data if $\alpha_\star \leq \alpha <\alpha^\star$,
\end{itemize}
where $\alpha_\star$ and $\alpha^\star$ are defined by
\begin{align}
\renewcommand*{\arraystretch}{1.2}
\alpha_\star:=\frac{4-2b}{d}, \quad \alpha^\star := \left\{ \begin{array}{c l}
\frac{4-2b}{d-2} & \text{if } d\geq 3, \\
\infty &\text{if } d=1, 2.
\end{array} \right. \label{alpha exponents}
\end{align}
\indent In the case $\alpha=\alpha_\star$ ($L^2$-critical), Genoud in \cite{Genoud} showed that the focusing $\eqref{INLS}$ with $0<b<\min\{2,d\}$ is globally well-posed in $H^1$ assuming $u_0 \in H^1$ and
\[
\|u_0\|_{L^2} <\|Q\|_{L^2},
\]
where $Q$ is the unique nonnegative, radially symmetric, decreasing solution of the ground state equation
\begin{align}
\Delta Q -Q +|x|^{-b}|Q|^{\frac{4-2b}{d}} Q=0. \label{ground state equation mass-critical}
\end{align}
Also, Combet-Genoud in \cite{CombetGenoud} established the classification of minimal mass blow-up solutions for the focusing $L^2$-critical $\eqref{INLS}$. \newline
\indent In the case $\alpha_\star <\alpha<\alpha^\star$, Farah in \cite{Farah} showed that the focusing $\eqref{INLS}$ with $0<b<\min \{2,d\}$ is globally well-posedness in $H^1, d\geq 1$ assuming $u_0 \in H^1$ and
\begin{align}
E(u_0)^{\Gact} M(u_0)^{1-\Gact} &< E(Q)^{\Gact} M(Q)^{1-\Gact}, \label{condition 1}\\
\|\nabla u_0\|^{\Gact}_{L^2} \|u_0\|_{L^2}^{1-\Gact} &< \|\nabla Q\|^{\Gact}_{L^2} \|Q\|^{1-\Gact}_{L^2}, \nonumber
\end{align}
where $Q$ is the unique nonnegative, radially symmetric, decreasing solution of the ground state equation
\begin{align}
\Delta Q- Q + |x|^{-b} |Q|^\alpha Q =0. \label{ground state equation introduction}
\end{align}
He also proved that if $u_0 \in H^1 \cap L^2(|x|^2dx)=:\Sigma$ satisfies $(\ref{condition 1})$ and 
\begin{align}
\|\nabla u_0\|^{\Gact}_{L^2} \|u_0\|_{L^2}^{1-\Gact} &> \|\nabla Q\|^{\Gact}_{L^2} \|Q\|^{1-\Gact}_{L^2},
\label{condition 2}
\end{align}
then the blow-up in $H^1$ must occur. Afterwards, Farah-Guzman in \cite{FarahGuzman1, FarahGuzman2} proved that the above global solution is scattering under the radial condition of the initial data. Note that the existence and uniqueness of solutions $Q$ to the elliptic equations $(\ref{ground state equation mass-critical})$ and $(\ref{ground state equation introduction})$ were proved by Toland \cite{Toland}, Yanagida \cite{Yanagida} and Genoud \cite{Genoud-unique} (see also Genoud-Stuart \cite{GenoudStuart}). \newline
\indent Guzman in \cite{Guzman} used Strichartz estimates and the contraction mapping argument to establish the local well-posedness as well as the small data global well-posedness for the $\eqref{INLS}$ in Sobolev space. Recently, the author in \cite{Dinhweighted} improved the local well-posedness in $H^1$ of Guzman by extending the validity of $b$ in the two and three dimensional spatial spaces. 
Note that the results of Guzman \cite{Guzman} and Dinh \cite{Dinhweighted} about the local well-posedness of $\eqref{INLS}$ in $H^1$ are a bit weaker than the one of Genoud-Stuart \cite{GenoudStuart}. More precisely, they do not treat the case $d=1$, and there is a restriction on the validity of $b$ when $d=2$ or $3$. However, the local well-posedness proved in \cite{Guzman, Dinhweighted} provides more information on the solutions, for instance, one knows that the global solutions to the defocusing $\eqref{INLS}$ satisfy $u \in L^p_{\text{loc}}(\R, W^{1.q})$ for any Schr\"odinger admissible pair $(p,q)$. This property plays an important role in proving the scattering for the $\eqref{INLS}$. Note also that the author in \cite{Dinhweighted} pointed out that one cannot expect a similar local well-posedness result for $\eqref{INLS}$ in $H^1$ as in \cite{Guzman, Dinhweighted} holds in the one dimensional case by using Strichartz estimates.   \newline
\indent In \cite{Dinhweighted}, the author used the so-called pseudo-conformal conservation law to show the decaying property of global solutions to the defocusing $\eqref{INLS}$ by assuming the initial data in $\Sigma$ (see before $(\ref{condition 2})$). In particular, he showed that in the case $\alpha \in [\alpha_\star, \alpha^\star)$, global solutions have the same decay as the solutions of the linear Schr\"odinger equation, that is for $2\leq q \leq \frac{2d}{d-2}$ when $d\geq 3$ or $2\leq q <\infty$ when $d=2$ or $2\leq q\leq\infty$ when $d=1$,
\[
\|u(t)\|_{L^q(\R^d)} \lesssim |t|^{-d\left(\frac{1}{2}-\frac{1}{q}\right)}, \quad \forall t\ne 0. 
\]
This allows the author proved the scattering in $\Sigma$ for a certain class of the defocusing $\eqref{INLS}$. Later, the author in \cite{Dinhenergy} made use of the classical Morawetz inequality and an argument of \cite{Visciglia} to derive the decay of global solutions to the defocusing $\eqref{INLS}$ with the initial data in $H^1$. Using the decaying property, he was able to show the energy scattering for a class of the defocusing $\eqref{INLS}$. We refer the reader to \cite{Dinhweighted, Dinhenergy} for more details. \newline
\indent The main purpose of this paper is to show the finite time blowup for the focusing $\eqref{INLS}$. Thanks to the well-posedness of Genoud-Stuart \cite{GenoudStuart}, we only expect blowup in $H^1$ when $\alpha_\star \leq \alpha < \alpha_\star$ which correspond to the mass-critical and the mass and energy intercritical cases. Note that the local well-posedness for the energy-critical $\eqref{INLS}$, i.e. $\alpha=\alpha^\star$ is still an open problem. \newline
\indent Our first result is the following finite time blowup for the $\eqref{INLS}$ in the mass-critical case $\alpha=\alpha_\star$.
\begin{theorem} \label{theorem blowup mass-critical}
Let $0<b<\min\{2,d\}$ and $u_0 \in H^1$. Then the corresponding solution to the focusing mass-critical $\eqref{INLS}$ blows up in finite time if one of the following conditions holds true:
\begin{itemize}
\item[1.] $d\geq 2, E(u_0)<0$ and $xu_0 \in L^2$,
\item[2.] $d \geq 2, E(u_0) <0$ and $u_0$ is radial,
\item[3.] $d=1$ and $E(u_0)<0$. 
\end{itemize}
\end{theorem}
\begin{rem} \label{rem blowup mass-critical}
\begin{itemize}
\item[1.] This theorem extends the well-known finite time blowup for the focusing mass-critical nonlinear Schr\"odinger equation (i.e. $b=0$ in $\eqref{INLS}$) \cite{OgawaTsutsumi1,OgawaTsutsumi2} to the focusing mass-critical $\eqref{INLS}$.  
\item[2.] The condition $E(u_0)<0$ is a sufficient condition for the finite time blowup but it is not necessary. In fact, one can show (see Remark $\ref{rem sufficient condition}$) that for any $E>0$, there exists $u_0 \in H^1$ satisfying $E(u_0) =E$ and the corresponding solution blows up in finite time.
\end{itemize}
\end{rem}
\indent We now consider the intercritical (i.e. mass-supercritical and energy-subcritical) case $\alpha_\star<\alpha<\alpha^\star$. Our next result is the following blowup for the intercritical $\eqref{INLS}$.
\begin{theorem} \label{theorem blowup intercritical}
Let  
\[
d\geq 3,\quad 0<b<2, \quad \alpha_\star<\alpha <\alpha^\star,
\]
or
\[
d=2, \quad 0<b<2, \quad \alpha_\star<\alpha \leq 4.
\]
Let $u_0 \in H^1$ be radial and satisfy either $E(u_0)<0$ or, if $E(u_0) \geq 0$, we suppose that
\begin{align}
E(u_0) M(u_0)^\sigma < E(Q) M(Q)^\sigma, \label{condition below ground state}
\end{align} 
and
\begin{align}
\|\nabla u_0\|_{L^2} \|u_0\|^\sigma_{L^2} > \|\nabla Q\|_{L^2} \|Q\|^\sigma_{L^2}, \label{condition blowup intercritical}
\end{align}
where 
\begin{align}
\sigma:=\frac{1-\Gact}{\Gact} = \frac{2(2-b)-(d-2)\alpha}{d\alpha-2(2-b)}, \label{define sigma}
\end{align}
and $Q$ is the unique solution the ground state equation $(\ref{ground state equation introduction})$. Then the corresponding solution to the focusing intercritical $\eqref{INLS}$ blows up in finite time. Moreover, if $u_0$ satisfies $(\ref{condition below ground state})$ and $(\ref{condition blowup intercritical})$, then the corresponding solution satisfies
\begin{align}
\|\nabla u(t)\|_{L^2} \|u(t)\|^\sigma_{L^2}> \|\nabla Q\|_{L^2} \|Q\|^\sigma_{L^2}, \label{property blowup intercritical}
\end{align}
for any $t$ in the existence time.
\end{theorem}
\begin{rem} \label{rem blowup intercritical}
\begin{itemize}
\item[1.] The restriction $\alpha \leq 4$ when $d=2$ is technical due to the Young inequality (see Lemma $\ref{lem localized virial identity}$). 
\item[1.] In \cite{Farah}, Farah proved the finite time blowup for the focusing intercritical $\eqref{INLS}$ for initial data $u_0 \in H^1 \cap L^2(|x|^2dx), d\geq 1$ satisfying $(\ref{condition below ground state})$ and $(\ref{condition blowup intercritical})$.
\item[2.] It was proved in \cite{Farah} that if the initial data $u_0$ satisfies $(\ref{condition below ground state})$ and $\|\nabla u_0\|_{L^2} \|u_0\|^\sigma_{L^2} < \|\nabla Q\|_{L^2} \|Q\|^\sigma_{L^2}$, then the corresponding solution exists globally in time.
\end{itemize}
\end{rem}
\indent This paper is organized as follows. In Section 2, we recall the sharp Gagliardo-Nirenberg inequality related to the focusing $\eqref{INLS}$ due to Farah \cite{Farah}. In Section 3, we derive the standard virial identity and localized virial estimates for the focusing $\eqref{INLS}$. We will give the proof of Theorem $\ref{theorem blowup mass-critical}$ in Section 4. Finally, the proof of Theorem $\ref{theorem blowup intercritical}$ will be given in Section 5.
\section{Sharp Gagliardo-Nirenberg inequality} \label{section sharp gagliardo-nirenberg inequality}
\setcounter{equation}{0}
In this section, we recall the sharp Gagliardo-Nirenberg inequality related to the focusing $\eqref{INLS}$ due to Farah \cite{Farah}. 
\begin{theorem}[Sharp Gagliardo-Nirenberg inequality \cite{Farah}] \label{theorem sharp gagliardo-nirenberg inequality}
Let $d\geq 1, 0<b<\min\{2,d\}$ and $0<\alpha<\alpha^\star$. Then the Gagliardo-Nirenberg inequality
\begin{align}
\int |x|^{-b} |u(x)|^{\alpha+2} dx \leq C_{\emph{GN}} \|u\|^{\frac{4-2b -(d-2)\alpha}{2}}_{L^2} \|\nabla u\|^{\frac{d\alpha+2b}{2}}_{L^2},\label{sharp gagliardo-nirenberg inequality}
\end{align}
holds true, and the sharp constant $C_{\emph{GN}}$ is attended by a function $Q$, i.e.
\begin{align}
C_{\emph{GN}} = \int |x|^{-b} |Q(x)|^{\alpha+2} dx \div \Big[\|Q\|^{\frac{4-2b -(d-2)\alpha}{2}}_{L^2} \|\nabla Q\|^{\frac{d\alpha+2b}{2}}_{L^2} \Big], \label{sharp constant definition}
\end{align}
where $Q$ is the unique non-negative, radially symmetric, decreasing solution to the elliptic equation
\begin{align}
\Delta Q - Q+ |x|^{-b} |Q|^\alpha Q = 0. \label{ground state equation}
\end{align}
\end{theorem}
\begin{rem} \label{rem sharp gagliardo-nirenberg inequality}
\begin{itemize}
\item[1.] In \cite{Farah}, Farah proved this result for $\alpha_\star<\alpha<\alpha^\star$. However, the proof and so the result are still valid for $0<\alpha\leq \alpha_\star$. 
\item[2.]  We also have the following Pohozaev identities:
\begin{align}
\|Q\|^2_{L^2} = \frac{4-2b-(d-2)\alpha}{d\alpha+2b} \|\nabla Q\|^2_{L^2} = \frac{4-2b-(d-2)\alpha}{2(\alpha+2)} \int |x|^{-b} |Q(x)|^{\alpha+2} dx. \label{pohozaev identities} 
\end{align}
In particular, 
\begin{align}
C_{\text{GN}}= \frac{2(\alpha+2)}{4-2b-(d-2)\alpha}\Big[\frac{4-2b-(d-2)\alpha}{d\alpha+2b} \Big]^{\frac{d\alpha+2b}{4}} \frac{1}{\|Q\|^\alpha_{L^2}}. \label{sharp constant gagliardo-nirenberg inequality}
\end{align}
\end{itemize}
\end{rem}
\section{Virial identities} \label{section virial identities}
\setcounter{equation}{0}
In this section, we derive virial identities and virial estimates related to the focusing $\eqref{INLS}$. Given a real valued function $a$, we define the virial potential by
\begin{align}
V_a(t):= \int a(x)|u(t,x)|^2 dx. \label{virial potential} 
\end{align}
By a direct computation, we have the following result (see e.g. \cite[Lemma 5.3]{TaoVisanZhang}.)
\begin{lem}[\cite{TaoVisanZhang}]  \label{lem derivative virial potential}
If $u$ is a smooth-in-time and Schwartz-in-space solution to 
\[
i\partial_t u +\Delta u = N(u),
\]
with $N(u)$ satisfying $\imemph{(N(u)\overline{u})}=0$, then we have
\begin{align}
\frac{d}{dt} V_a (t)= 2 \int_{\R^d}\nabla a(x) \cdot  \imemph	{(\overline{u}(t,x) \nabla u(t,x))} dx, \label{first derivative viral potential}
\end{align}
and
\begin{equation}
\begin{aligned}
\frac{d^2}{dt^2} V_a(t) = -\int \Delta^2 a(x) |u(t,x)|^2  dx & +  4 \sum_{j,k=1}^d \int \partial^2_{jk} a(x) \reemph{(\partial_k u(t,x) \partial_j \overline{u}(t,x))} dx  \\
&+ 2\int \nabla a(x)\cdot \{N(u), u\}_p(t,x) dx, 
\end{aligned} \label{second derivative virial potential} 
\end{equation}
where $\{f, g\}_p :=\reemph{(f\nabla \overline{g} - g \nabla \overline{f})}$ is the momentum bracket.
\end{lem} 
We note that if $N(u)=-|x|^{-b} |u^\alpha u$, then 
\[
\{N(u), u\}_p = \frac{\alpha}{\alpha+2} \nabla(|x|^{-b} |u|^{\alpha+2}) +\frac{2}{\alpha+2} \nabla(|x|^{-b}) |u|^{\alpha+2}.
\] 
Using this fact, we immediately have the following result.
\begin{coro}\label{coro derivative virial potential}
If $u$ is a smooth-in-time and Schwartz-in-space solution to the focusing $\eqref{INLS}$, then we have 
\begin{equation}
\begin{aligned}
\frac{d^2}{dt^2} V_a(t) &= -\int \Delta^2 a(x) |u(t,x)|^2  dx  +  4 \sum_{j,k=1}^d \int \partial^2_{jk} a(x) \reemph{(\partial_k u(t,x) \partial_j \overline{u}(t,x))} dx  \\
&\mathrel{\phantom{=}}-\frac{2\alpha}{\alpha+2} \int \Delta a(x) |x|^{-b} |u(t,x)|^{\alpha+2} dx +\frac{4}{\alpha+2} \int \nabla a(x) \cdot \nabla(|x|^{-b}) |u(t,x)|^{\alpha+2} dx. 
\end{aligned} \label{second derivative virial potential application} 
\end{equation}
\end{coro}
A direct consequence of Corollary $\ref{coro derivative virial potential}$ is the following standard virial identity for the $\eqref{INLS}$. 
\begin{lem} \label{lem standard virial identity}
Let $u_0 \in H^1$ be such that $|x|u_0 \in L^2$ and $u:I \times \R^d \rightarrow \C$ the corresponding solution to the focusing $\eqref{INLS}$. Then, $|x| u \in C(I, L^2)$. Moreover, for any $t\in I$,
\begin{align}
\frac{d^2}{dt^2} \|x u(t)\|^2_{L^2_x} = 8\|\nabla u(t)\|^2_{L^2} - \frac{4(d\alpha+2b)}{\alpha+2}\int |x|^{-b}|u(t,x)|^{\alpha+2} dx. \label{standard virial identity}
\end{align}
\end{lem}
\begin{proof}
The first claim follows from the standard approximation argument, we omit the proof and refer the reader to \cite[Proposition 6.5.1]{Cazenave} for more details. The identity $(\ref{standard virial identity})$ follows from Corollary $\ref{coro derivative virial potential}$ by taking $a(x)=|x|^2$.
\end{proof}
In order to prove the blowup for the focusing $\eqref{INLS}$ with radial data, we need localized virial estimates. To do so, we introduce a function $\theta: [0,\infty) \rightarrow [0,\infty)$ satisfying
\begin{align}
\theta(r) =\left\{\begin{array}{cl}
r^2 & \text{if } 0\leq r \leq 1, \\
\text{const.} &\text{if } r \geq 2, 
\end{array}
\right.
\quad \text{and} \quad \theta''(r) \leq 2 \quad \text{for }  r\geq 0.  \label{define theta}
\end{align}
Note that the precise constant here is not important. For $R>1$, we define the radial function
\begin{align}
\varphi_R(x) = \varphi_R(r):=R^2 \theta(r/R), \quad r=|x|. \label{define rescaled varphi}
\end{align}
It is easy to see that
\begin{align}
2-\varphi''_R(r) \geq 0, \quad 2-\frac{\varphi'_R(r)}{r} \geq 0, \quad 2d- \Delta \varphi_R(x) \geq 0. \label{property rescaled varphi}
\end{align}
\begin{lem} \label{lem localized virial identity}
Let $d\geq 2, 0<b<2, 0<\alpha \leq 4, R>1$ and $\varphi_R$ be as in $(\ref{define rescaled varphi})$. Let $u: I\times \R^d \rightarrow \C$ be a radial solution to the focusing $\eqref{INLS}$. Then for any $\ep>0$ and any $t\in I$,
\begin{align}
\begin{aligned}
\frac{d^2}{dt^2}V_{\varphi_R} (t) &\leq 8\|\nabla u(t)\|^2_{L^2} - \frac{4(d\alpha+2b)}{\alpha+2}\int |x|^{-b}|u(t,x)|^{\alpha+2}dx \\ 
& \mathrel{\phantom{\leq}} + \left\{
\begin{array}{cl}
O\left( R^{-2} + R^{-[2(d-1) +b]} \|\nabla u(t)\|^2_{L^2} \right) &\text{if } \alpha =4, \\
O \left( R^{-2} + \ep^{-\frac{\alpha}{4-\alpha}} R^{-\frac{2[(d-1)\alpha+2b]}{4-\alpha}} + \ep \|\nabla u(t)\|^2_{L^2} \right) &\text{if } \alpha<4.
\end{array}
\right. 
\end{aligned}
\label{localized virial identity}
\end{align}
\end{lem}
\begin{rem} \label{rem localized virial identity}
\begin{itemize}
\item[1.] The condition $d\geq 2$ comes from the radial Sobolev embedding. This is due to the fact that radial functions in dimension 1 do not have any decaying property. The restriction $0<\alpha \leq 4$ comes from the Young inequality below.
\item[2.] If we consider $\alpha_\star \leq \alpha \leq \alpha^\star$, then there is a restriction on the validity of $\alpha$ in 2D. More precisely, we need $\alpha_\star \leq \alpha \leq 4$ when $d=2$.
\end{itemize}
\end{rem}
\noindent \textit{Proof of Lemma $\ref{lem localized virial identity}$.} We apply $(\ref{second derivative virial potential application})$ for $a(x) =\varphi_R(x)$ to get
\begin{align*}
\frac{d^2}{dt^2} V_{\varphi_R}(t) =& -\int \Delta^2 \varphi_R |u(t)|^2  dx  +  4 \sum_{j,k=1}^d \int \partial^2_{jk} \varphi_R \re{(\partial_k u(t) \partial_j \overline{u}(t))} dx  \\
&-\frac{2\alpha}{\alpha+2} \int \Delta \varphi_R |x|^{-b} |u(t)|^{\alpha+2} dx +\frac{4}{\alpha+2} \int \nabla\varphi_R \cdot \nabla(|x|^{-b}) |u(t)|^{\alpha+2}dx. 
\end{align*}
Since $\varphi_R(x)=|x|^2$ for $|x|\leq R$, we use $(\ref{standard virial identity})$ to have
\begin{align}
\frac{d^2}{dt^2}V_{\varphi_R} (t)&=8\|\nabla u(t)\|^2_{L^2} - \frac{4(d\alpha+2b)}{\alpha+2}\int |x|^{-b}|u(t)|^{\alpha+2} dx \nonumber \\
&\mathrel{\phantom{=}} - 8\|\nabla u(t)\|^2_{L^2(|x|>R)} + \frac{4(d\alpha+2b)}{\alpha+2} \int_{|x|>R}|x|^{-b}|u(t)|^{\alpha+2}dx \label{localized virial estimate proof} \\
&\mathrel{\phantom{=}} -\int_{|x|>R} \Delta^2 \varphi_R |u(t)|^2  dx  +  4 \sum_{j,k=1}^d \int_{|x|>R} \partial^2_{jk} \varphi_R \re{(\partial_k u(t) \partial_j \overline{u}(t))} dx \nonumber \\
&\mathrel{\phantom{=}} -\frac{2\alpha}{\alpha+2}\int_{|x|>R} \Delta \varphi_R |x|^{-b} |u(t)|^{\alpha+2} dx+\frac{4}{\alpha+2} \int_{|x|>R} \nabla \varphi_R \cdot \nabla(|x|^{-b}) |u(t)|^{\alpha+2} dx. \nonumber
\end{align}
Since $|\Delta \varphi_R |\lesssim 1, |\Delta^2 \varphi_R| \lesssim R^{-2}$ and $|\nabla \varphi_R \cdot \nabla(|x|^{-b})| \lesssim |x|^{-b}$, we have
\begin{align*}
\frac{d^2}{dt^2}V_{\varphi_R} (t)= 8\|\nabla u(t)\|^2_{L^2} &- \frac{4(d\alpha+2b)}{\alpha+2}\int |x|^{-b}|u(t)|^{\alpha+2} dx  \\
&+ 4\sum_{j,k=1}^d \int_{|x|>R} \partial^2_{jk} \varphi_R \re{(\partial_k u(t) \partial_j \overline{u}(t))} dx - 8\|\nabla u(t)\|^2_{L^2(|x|>R)} \\
& + O \Big( \int_{|x|> R} R^{-2} |u(t)|^2 + |x|^{-b} |u(t)|^{\alpha+2} dx \Big).
\end{align*}
Using $(\ref{property rescaled varphi})$ and the fact that
\[
\partial^2_{jk} = \Big(\frac{\delta_{jk}}{r}-\frac{x_j x_k}{r^3}\Big) \partial_r + \frac{x_j x_k}{r^2} \partial^2_r,
\]
we see that
\[
\sum_{j,k=1}^d \partial^2_{jk}\varphi_R \partial_k u \partial_j \overline{u} = \varphi''_R(r) |\partial_r u|^2 \leq 2 |\partial_r u|^2 = 2|\nabla u|^2.
\]
Therefore
\[
4\sum_{j,k=1}^d \int_{|x|>R} \partial^2_{jk} \varphi_R \re{(\partial_k u (t) \partial_j \overline{u}(t))} dx - 8\|\nabla u(t)\|^2_{L^2(|x|>R)} \leq 0.
\]
The conservation of mass then implies
\begin{align*}
\frac{d^2}{dt^2} V_{\varphi_R}(t) \leq 8\|\nabla u(t)\|^2_{L^2} - \frac{4(d\alpha+2b)}{\alpha+2}\int |x|^{-b}|u(t)|^{\alpha+2} dx + O \Big( R^{-2} + \int_{|x|>R}|x|^{-b} |u(t)|^{\alpha+2} dx \Big).
\end{align*}
It remains to bound $\int_{|x|>R}|x|^{-b} |u(t)|^{\alpha+2} dx$. To do this, we recall the following radial Sobolev embedding (\cite{Strauss, ChoOzawa}).
\begin{lem}[Radial Sobolev embedding \cite{Strauss, ChoOzawa}]
Let $d\geq 2$ and $\frac{1}{2}\leq s <1$. Then for any radial function $f$, 
\begin{align}
\sup_{x\ne 0} |x|^{\frac{d-2s}{2}}|f(x)| \leq C(d,s) \|f\|^{1-s}_{L^2} \|f\|^s_{\dot{H}^1}. \label{usual radial sobolev embedding}
\end{align}
Moreover, the above inequality also holds for $d\geq 3$ and $s=1$.
\end{lem}
Using $(\ref{usual radial sobolev embedding})$ with $s=\frac{1}{2}$ and the conservation of mass, we estimate
\begin{align*}
\int_{|x|>R}|x|^{-b} |u(t)|^{\alpha+2} dx &\leq \Big(\sup_{|x|>R} |x|^{-b}|u(t,x)|^\alpha \Big)\|u(t)\|^2_{L^2} \\
&\lesssim R^{-\left[\frac{(d-1)\alpha}{2}+b\right]} \Big(\sup_{|x|>R} |x|^{\frac{d-1}{2}} |u(t,x)| \Big)^{\alpha} \|u(t)\|^2_{L^2} \\
&\lesssim R^{-\left[\frac{(d-1)\alpha}{2}+b\right]} \|\nabla u(t)\|^{\frac{\alpha}{2}}_{L^2} \|u(t)\|^{\frac{\alpha}{2}+2}_{L^2} \\
&\lesssim R^{-\left[\frac{(d-1)\alpha}{2}+b\right]} \|\nabla u(t)\|^{\frac{\alpha}{2}}_{L^2}.
\end{align*}
When $\alpha=4$, we are done. Let us consider $0<\alpha<4$. To do so, we recall the Young inequality: for $a, b$ non-negative real numbers and $p, q$ positive real numbers satisfying $\frac{1}{p}+\frac{1}{q}=1$, then for any $\ep>0$, $ab \lesssim \ep a^p + \ep^{-\frac{q}{p}} b^q$. Applying the Young inequality for $a=\|\nabla u(t)\|^{\frac{\alpha}{2}}_{L^2}, b= R^{-\left[\frac{(d-1)\alpha}{2}+b\right]}$ and $p=\frac{4}{\alpha}, q= \frac{4}{4-\alpha}$, we get for any $\ep>0$,
\[
R^{-\left[\frac{(d-1)\alpha}{2}+b\right]} \|\nabla u(t)\|^{\frac{\alpha}{2}}_{L^2} \lesssim \ep \|\nabla u(t)\|^2_{L^2} + \ep^{-\frac{\alpha}{4-\alpha}} R^{-\frac{2[(d-1)\alpha]+2b}{4-\alpha}}.
\]
Note that the condition $0<\alpha<4$ ensures $1<p, q<\infty$. The proof is complete.
\defendproof
\newline
\indent In the mass-critical case $\alpha=\alpha_\star$, we have the following refined version of Lemma $\ref{lem localized virial identity}$. The proof of this result is based on an argument of \cite{OgawaTsutsumi1} (see also \cite{BoulengerHimmelsbachLenzmann}). 
\begin{lem} \label{lem localized virial identity mass-critical}
Let $d\geq 2, 0<b<2,  R>1$ and $\varphi_R$ be as in $(\ref{define rescaled varphi})$. Let $u: I\times \R^d \rightarrow \C$ be a radial solution to the focusing mass-critical $\eqref{INLS}$, i.e. $\alpha=\alpha_\star$. Then for any $\ep>0$ and any $t\in I$,
\begin{align}
\begin{aligned}
\frac{d^2}{dt^2}V_{\varphi_R} (t)\leq 16E(u_0) &-2\int_{|x|>R} \Big(\chi_{1,R} - \frac{\ep}{d+2-b} \chi_{2,R}^{\frac{d}{2-b}} \Big) |\nabla u(t)|^2dx \\
&+ O \Big( R^{-2} +\ep R^{-2} + \ep^{-\frac{2-b}{2d-2+b}}R^{-2}\Big),
\end{aligned}
\label{localized virial identity mass-critical}
\end{align}
where
\begin{align}
\chi_{1,R} = 2(2-\varphi''_R), \quad \chi_{2,R}= (2-b)(2d-\Delta \varphi_R) +db \Big(2-\frac{\varphi'_R}{r}\Big). \label{define chi}
\end{align}
\end{lem}
\begin{proof}
We first notice that
\[
\sum_{j,k} \partial^2_{jk} \varphi_R \partial_k u \partial_j \overline{u} = \varphi''_R |\partial_r u|^2, \quad \nabla\varphi_R \cdot \nabla(|x|^{-b}) = -b \frac{\varphi'_R}{r}|x|^{-b}.
\]
Using $(\ref{localized virial estimate proof})$ with $\alpha=\alpha_\star =\frac{4-2b}{d}$ and rewriting $\varphi''_R = 2 - (2-\varphi''_R), \frac{\varphi'_R}{r} = 2-\left(2-\frac{\varphi'_R}{r}\right)$ and $\Delta \varphi_R = 2d - (2d-\Delta \varphi_R)$, we have
\begin{align*}
\frac{d^2}{dt^2}V_{\varphi_R}(t)&= 16E(u(t)) -\int_{|x|>R} \Delta^2\varphi_R |u(t)|^2 dx - 4\int_{|x|>R} (2-\varphi''_R) |\partial_ru(t)|^2dx  \\
& \mathrel{\phantom{=16E(u(t))}} + \frac{4-2b}{d+2-b} \int_{|x|>R} (2d-\Delta\varphi_R) |x|^{-b} |u(t)|^{\frac{4-2b}{d}+2} dx \\
& \mathrel{\phantom{=16E(u(t))}} + \frac{2db}{d+2-b} \int_{|x|>R} \Big(2-\frac{\varphi'_R}{r}\Big) |x|^{-b} |u(t)|^{\frac{4-2b}{d}+2} dx \\
&\leq 16E(u(t)) + O(R^{-2}) - 2 \int_{|x|>R} \chi_{1,R} |\nabla u(t)|^2 dx \\
&\mathrel{\phantom{\leq 16E(u(t)) + O(R^{-2})}}+ \frac{2}{d+2-b}  \int_{|x|>R} \chi_{2,R} |x|^{-b} |u(t)|^{\frac{4-2b}{d}+2}dx,
\end{align*}
where $\chi_{1,R}$ and $\chi_{2,R}$ are as in $(\ref{define chi})$. Using the radial Sobolev embedding $(\ref{usual radial sobolev embedding})$ with $s=\frac{1}{2}$, the conservation of mass and the fact $|\chi_{2,R}| \lesssim 1$, we estimate
\begin{align*}
\int_{|x|>R} \chi_{2,R} |x|^{-b} |u(t)|^{\frac{4-2b}{d}+2} dx & = \int_{|x|>R} |x|^{-b} \Big|\chi_{2,R}^{\frac{d}{4-2b}} u(t)\Big|^{\frac{4-2b}{d}} |u(t)|^2 dx \\
&\leq \Big(\sup_{|x|>R} |x|^{-b} \Big|\chi^{\frac{d}{4-2b}}_{2,R}(x) u(t,x)\Big|^{\frac{4-2b}{d}}\Big) \|u(t)\|^2_{L^2} \\
&\leq R^{-\Big[\frac{(2-b)(d-1)}{d} + b\Big]} \Big( \sup_{|x|>R} |x|^{\frac{d-1}{2}} \Big|\chi^{\frac{d}{4-2b}}_{2,R}(x) u(t,x)\Big| \Big)^{\frac{4-2b}{d}} \|u(t)\|^2_{L^2}  \\
& \lesssim R^{-\Big[\frac{(2-b)(d-1)}{d} + b\Big]} \Big\| \nabla\Big(\chi_{2,R}^{\frac{d}{4-b}} u(t) \Big) \Big\|^{\frac{2-b}{d}}_{L^2} \Big\| \chi_{2,R}^{\frac{d}{4-b}} u(t)\Big\|^{\frac{2-b}{d}}_{L^2}\|u(t)\|^{2}_{L^2} \\
&\lesssim R^{-\Big[\frac{(2-b)(d-1)}{d} + b\Big]} \Big\| \nabla\Big(\chi_{2,R}^{\frac{d}{4-2b}} u(t) \Big) \Big\|^{\frac{2-b}{d}}_{L^2}.
\end{align*}
We next apply the Young inequality with $p= \frac{2d}{2-b}$ and $q=\frac{2d}{2d-2+b}$ to get for any $\ep>0$
\[
R^{-\Big[\frac{(2-b)(d-1)}{d} + b\Big]} \Big\| \nabla\Big(\chi_{2,R}^{\frac{d}{4-2b}} u(t) \Big) \Big\|^{\frac{2-b}{d}}_{L^2} \lesssim \ep \Big\| \nabla\Big(\chi_{2,R}^{\frac{d}{4-2b}} u(t) \Big) \Big\|^2_{L^2} +  \ep^{-\frac{2-b}{2d-2+b}} R^{-2}.
\] 
Moreover, using $(\ref{define theta}), (\ref{define rescaled varphi})$ and $(\ref{property rescaled varphi})$, it is easy to check that $|\nabla(\chi_{2,R}^{d/(4-2b)})| \lesssim R^{-1}$. Thus the conservation of mass implies
\[
\Big\| \nabla\Big(\chi_{2,R}^{\frac{d}{4-2b}} u(t) \Big) \Big\|^2_{L^2} \lesssim R^{-2} + \Big\|\chi_{2,R}^{\frac{d}{4-2b}} \nabla u(t) \Big\|^2_{L^2}.
\]
Combining the above estimates, we prove $(\ref{localized virial identity mass-critical})$.
\end{proof}
To prove the blowup in the 1D mass-critical case $\alpha=4-2b$, we need the following version of localized virial estimates due to \cite{OgawaTsutsumi2}. Let $\vartheta$ be a real-valued function in $W^{3,\infty}$ satisfying
\begin{align}
\vartheta(x)= \left\{
\begin{array}{c l c}
2x &\text{if}& 0\leq |x| \leq 1, \\
2[x-(x-1)^3] &\text{if}& 1<x \leq 1+1/\sqrt{3}, \\
2[x-(x+1)^3] &\text{if}& -(1+1/\sqrt{3}) \leq x <-1, \\
\vartheta'<0 &\text{if}& 1+1/\sqrt{3} <|x| <2,\\
0 &\text{if}& |x|\geq 2.
\end{array}
\right. \label{define vartheta 1D}
\end{align}  
Set 
\begin{align}
\theta(x) = \int_0^x \vartheta(s) ds. \label{define theta 1D}
\end{align}
\begin{lem} \label{lem virial estimate 1D}
Let $0<b<1$ and $\theta$ be as in $(\ref{define theta 1D})$. Let $u: I \times \R \rightarrow \C$ be a solution to the focusing mass-critical $\eqref{INLS}$, i.e. $\alpha=4-2b$. There exists $a_0>0$ such that if 
\begin{align}
\|u(t)\|_{L^2(|x|>1)} \leq a_0, \label{condition a_0}
\end{align}
for any $t\in I$, then there exists $C>0$ such that
\begin{align}
\frac{d^2}{dt^2} V_\theta(t) \leq 16E(u_0) + C(1+N)^{2-b} \|u(t)\|^{6-2b}_{L^2(|x|>1)} + N \|u(t)\|^2_{L^2(|x|>1)}, \label{virial estimate 1D}
\end{align}
for any $t\in I$, where $N:= \|\partial_x \vartheta\|_{L^\infty}+\|\partial^2_x \vartheta\|_{L^\infty} + \|\partial^3_x \vartheta\|_{L^\infty}$.
\end{lem}
\begin{proof}
We apply $(\ref{second derivative virial potential application})$ with $a(x)=\theta(x)$ to get
\begin{align*}
\frac{d^2}{dt^2}V_\theta(t) &= -\int \partial^4_x \theta |u(t)|^2 dx + 4\int \partial^2_x \theta |\partial_x u(t)|^2 dx -\frac{4-2b}{3-b}\int \partial^2_x \theta |x|^{-b} |u(t)|^{6-2b} dx  \\
&\mathrel{\phantom{= -\int \partial^4_x \theta |u(t)|^2 dx + 4\int \partial^2_x \theta |\partial_x u(t)|^2 dx}}+ \frac{2}{3-b} \int \partial_x \theta \partial_x(|x|^{-b}) |u(t)|^{6-2b} dx.
\end{align*}
Since $\theta(x)=x^2$ on $|x|\leq 1$, the definition of energy implies
\begin{align}
\frac{d^2}{dt^2}V_\theta(t) &= 16E(u(t)) -\int_{|x|>1} \partial^4_x \theta |u(t)|^2 dx - 4\int_{|x|>1} (2-\partial^2_x\theta) |\partial_x u(t)|^2 dx \nonumber \\
&\mathrel{\phantom{=}}+\frac{4-2b}{3-b}\int_{|x|>1} (2-\partial^2_x \theta) |x|^{-b} |u(t)|^{6-2b} dx +\frac{2b}{3-b} \int_{|x|>1}\Big(2-\frac{\partial_x \theta}{x}\Big) |x|^{-b}|u(t)|^{6-2b}dx \nonumber \\
&= 16E(u_0) -\int_{|x|>1} \partial^4_x \theta |u(t)|^2 dx - 2 \int_{|x|>1} \chi_1 |\partial_x u(t)|^2 dx  \nonumber \\
&\mathrel{\phantom{=16E(u_0) -\int_{|x|>1} \partial^4_x \theta |u(t)|^2 dx}} +\frac{2}{3-b}\int_{|x|>1}\chi_2 |x|^{-b} |u(t)|^{6-2b} dx, \label{virial estimate 1D proof}
\end{align}
where
\[
\chi_1 := 2(2-\partial^2_x \theta), \quad \chi_2:= (2-b)(2-\partial^2_x\theta) + b\Big(2-\frac{\partial_x \theta}{x}\Big).
\]
We now estimate 
\begin{align*}
\int_{|x|>1} \chi_2 |x|^{-b}|u(t)|^{6-2b} dx & \leq \Big(\sup_{|x|>1} \chi_2 |x|^{-b} |u(t)|^{4-2b} \Big) \|u(t)\|^2_{L^2(|x|>1)} \\
& \leq \|\rho u(t)\|^{4-2b}_{L^\infty(|x|>1)} \|u(t)\|^2_{L^2(|x|>1)},
\end{align*}
where $\rho(x):= \chi_2^{\frac{1}{4-2b}}(x)$. Using a variant of the Gagliardo-Nirenberg inequality (see e.g. \cite[Lemma 2.1]{OgawaTsutsumi2}), we bound
\[
\|\rho u(t)\|_{L^\infty(|x|>1)}  \leq \|u(t)\|^{1/2}_{L^2(|x|>1)} \Big[2 \|\rho^2 \partial_x u(t)\|_{L^2(|x|>1)} + \|u(t)\partial_x(\rho^2)\|_{L^2(|x|>1)} \Big]^{1/2}.
\] 
Thus,
\begin{align}
\int_{|x|>1} \chi_2 |x|^{-b} |u(t)|^{6-2b} dx &\leq \|u(t)\|^{4-b}_{L^2(|x|>1)} \Big[ 2 \|\rho^2 \partial_x u(t)\|_{L^2(|x|>1)} + \|u(t)\partial_x(\rho^2)\|_{L^2(|x|>1)} \Big]^{2-b} \nonumber \\
& \leq \|u(t)\|^{4-b}_{L^2(|x|>1)} 2^{1-b} \Big[ 2^{2-b} \|\rho^2 \partial_x u(t)\|^{2-b}_{L^2(|x|>1)} + \|u(t)\partial_x(\rho^2)\|^{2-b}_{L^2(|x|>1)}\Big] \nonumber \\
& \leq 2^{3-2b}\|u(t)\|^{4-b}_{L^2(|x|>1)}\|\rho^2 \partial_x u(t)\|^{2-b}_{L^2(|x|>1)} \nonumber \\
& \mathrel{\phantom{\leq 2^{3-2b}\|u(t)\|^{4-b}_{L^2(|x|>1)}}}+ 2^{1-b} \|u(t)\|^{6-2b}_{L^2(|x|>1)} \|\partial_x(\rho^2)\|^{2-b}_{L^\infty(|x|>1)}. \label{estimate nonlinear term 1D}
\end{align}
We next estimate $\|\partial_x(\rho^2)\|_{L^\infty(|x|>1)}$. By the definition of $\rho$, we write
\[
\partial_x(\rho^2) = \frac{1}{2-b} \frac{\partial_x\chi_2}{\chi_2^{\frac{1-b}{2-b}}}.
\]
\indent On $1<|x|\leq 1+1/\sqrt{3}$, a direct computation shows
\[
2-\frac{\partial_x \theta}{x} = 2\frac{(|x|-1)^3}{|x|}, \quad 2-\partial^2\theta = 6(|x|-1)^2.
\]
Thus,
\[
\chi_2 = 6(2-b) (|x|-1)^2 + 2b\frac{(|x|-1)^3}{|x|},
\]
and
\[
\renewcommand*{\arraystretch}{1.5}
\partial_x \chi_2 = \left\{
\begin{array}{llc}
(x-1)\left[12(2-b)   + 2b\frac{(x-1)(3x-1)}{x^2}\right] &\text{if}& 1<x\leq 1+1/\sqrt{3}, \\
(x+1)\left[12(2-b)  + 2b\frac{(x+1)(3x-1)}{x^2}\right] &\text{if}& -(1+1/\sqrt{3})\leq x <-1.
\end{array}
\right.
\]
Thus
\[
\renewcommand*{\arraystretch}{2}
\frac{\partial_x \chi_2}{\chi_2^{\frac{1-b}{2-b}}} = \left\{
\begin{array}{llc}
(x-1)^{\frac{b}{2-b}}\frac{\left[12(2-b)   + 2b\frac{(x-1)(3x-1)}{x^2}\right]}{\left[6(2-b) + 2b\frac{x-1}{x} \right]^{\frac{1-b}{2-b}}} &\text{if}& 1<x\leq 1+1/\sqrt{3}, \\
(x+1)^{\frac{b}{2-b}}\frac{\left[12(2-b)   + 2b\frac{(x+1)(3x-1)}{x^2}\right]}{\left[6(2-b) + 2b\frac{x+1}{x} \right]^{\frac{1-b}{2-b}}}&\text{if}& -(1+1/\sqrt{3})\leq x <-1.
\end{array}
\right.
\]
This implies that $\partial_x \chi_2/\chi_2^{\frac{1-b}{2-b}}$ is uniformly bounded on $1<|x| \leq 1+1/\sqrt{3}$. \newline
\indent On $|x|>1+1/\sqrt{3}$, we note that $\chi_2 \geq 4$ since $\partial^2_x \theta$ and $\partial_x \theta/x$ are both non-positive there by the choice of $\vartheta$. We thus simply bound
\[
\Big|\partial_x\chi_2/\chi_2^{\frac{1-b}{2-b}}\Big| \lesssim \|\partial_x \vartheta\|_{L^\infty} + \|\partial^2_x \vartheta\|_{L^\infty} \lesssim N. 
\] 
Therefore,
\[
\|\partial_x(\rho^2)\|_{L^\infty(|x|>1)} \lesssim 1+N.
\]
Combining this with $(\ref{estimate nonlinear term 1D})$, we obtain
\begin{align}
\int_{|x|>1}\chi_2 |x|^{-b} |u(t)|^{6-2b} dx \leq 2^{3-2b} \|u(t)\|^{4-b}_{L^2(|x|>1)} \|\rho^2 \partial_x u(t)\|^{2-b}_{L^2(|x|>1)} + C(1+N)^{2-b} \|u(t)\|^{6-2b}_{L^2(|x|>1)}, \label{estimate nonlinear term 1D result}
\end{align}
for some constant $C>0$. We thus get from $(\ref{virial estimate 1D proof})$ and $(\ref{estimate nonlinear term 1D result})$ that
\begin{align*}
\frac{d^2}{dt^2} V_\theta (t) &\leq  16E(u_0) + N \|u(t)\|^2_{L^2(|x|>1)} - 2\int_{|x|>1} \chi_1 |\partial_x u(t)|^2 dx \\
&\mathrel{\phantom{\leq  16E(u_0)}}+\frac{2^{4-2b}}{3-b} \|u(t)\|^{4-b}_{L^2(|x|>1)} \|\rho^2 \partial_x u(t)\|^{2-b}_{L^2(|x|>1)} + C(1+N)^{2-b} \|u(t)\|^{6-2b}_{L^2(|x|>1)} \\
&\leq 16E(u_0) - 2\int_{|x|>1} \Big(\chi_1 -\frac{2^{3-2b}}{3-b} \chi_2 \|u(t)\|^{4-b}_{L^2(|x|>1)} \Big) |\partial_x u(t)|^2 dx   \\
&\mathrel{\phantom{\leq  16E(u_0)}} + C(1+N)^{2-b} \|u(t)\|^{6-2b}_{L^2(|x|>1)} + N \|u(t)\|^2_{L^2(|x|>1)}.
\end{align*} 
We will show that if $\|u(t)\|_{L^2(|x|>1)}\leq a_0$ for some $a_0>0$ small enough, then
\begin{align}
\chi_1- \frac{2^{3-2b}}{3-b} \chi_2 \|u(t)\|^{4-b}_{L^2(|x|>1)} \geq 0, \label{condition smallness 1D}
\end{align}
for any $|x|>1$. It immediately yields $(\ref{virial estimate 1D})$. It remains to prove $(\ref{condition smallness 1D})$. To do so, it is enough to show for some $a_1>0$ small enough,
\begin{align}
\chi_1- a_1\frac{2^{3-2b}}{3-b}  \chi_2  \geq 0,  \label{condition smallness 1D equivalence}
\end{align}
for any $|x|>1$. \newline
\indent On $1<|x|\leq 1+1/\sqrt{3}$, we have
\[
\chi_1=2(2-\partial^2_x \theta) = 12(|x|-1)^2, 
\]
and
\begin{align*}
\chi_2 =(2-b) (2-\partial^2_x\theta) + b\Big(2-\frac{\partial_x\theta}{x}\Big) &= 6(2-b)(|x|-1)^2 +2b \frac{(|x|-1)^3}{|x|} \\
&= 6(|x|-1)^2 \Big[2-b + b\frac{|x|-1}{3|x|}\Big]<6(|x|-1)^2 \Big[2-b +\frac{b}{3\sqrt{3}}\Big].
\end{align*}
Thus, by taking $a_1>0$ small enough, we have $(\ref{condition smallness 1D equivalence})$. \newline
\indent On $|x|>1+1/\sqrt{3}$, since $\partial^2_x \theta = \partial_x \vartheta \leq 0$, we have $\chi_1 \geq 4$. Moreover, $\chi_2 \leq C$ for some constant $C>0$. We thus get $(\ref{condition smallness 1D equivalence})$ by taking $a_1>0$ small enough. The proof is complete.
\end{proof}
\section{Mass-critical case $\alpha=\alpha_\star$} 
\setcounter{equation}{0}
In this section, we will give the proof of Theorem $\ref{theorem blowup mass-critical}$. 
\subsection{The case $d\geq 1, E(u_0)<0$ and $xu_0 \in L^2$} Applying $(\ref{standard virial identity})$ with $\alpha=\alpha_\star$, we see that
\[
\frac{d^2}{dt^2} \|xu(t)\|^2_{L^2} = 8 \|\nabla u(t)\|^2_{L^2} -\frac{16}{\alpha_\star+2} \int |x|^{-b} |u(t,x)|^{\alpha_\star+2}dx = 16E(u_0)<0.
\]
By the classical argument of Glassey \cite{Glassey}, the solution must blow up in finite time.
\subsection{The case $d\geq 2, E(u_0)<0$ and $u_0$ is radial}
We use the localized virial estimate $(\ref{localized virial identity mass-critical})$ to have
\begin{align*}
\frac{d^2}{dt^2}V_{\varphi_R} (t)&\leq 16E(u_0) -2\int_{|x|>R} \Big(\chi_{1,R} - \frac{\ep}{d+2-b} \chi_{2,R}^{\frac{d}{2-b}} \Big) |\nabla u(t)|^2dx \\
&\mathrel{\phantom{\leq 16E_c(u_0)}}+ O \Big( R^{-2} + \ep R^{-2} +\ep^{-\frac{2-b}{2d-2+b}}R^{-2} \Big), 
\end{align*}
where 
\[
\chi_{1,R} = 2(2-\varphi''_R), \quad \chi_{2,R}= (2-b)(2d-\Delta \varphi_R) + db\Big(2-\frac{\varphi'_R}{r}\Big). 
\]
If we choose a suitable radial function $\varphi_R$ defined by $(\ref{define rescaled varphi})$ so that
\begin{align}
\chi_{1,R} - \frac{\ep}{d+2-b} \chi_{2,R}^{\frac{d}{2-b}} \geq 0, \quad \forall r>R, \label{positive condition}
\end{align}
for a sufficiently small $\ep>0$, then by choosing $R>1$ sufficiently large depending on $\ep$, we see that
\[
\frac{d^2}{dt^2}V_{\varphi_R} (t)\leq 8E(u_0)<0,
\]
for any $t$ in the existence time. This shows that the solution $u$ blows up in finite time. It remains to find $\varphi_R$ so that $(\ref{positive condition})$ holds true. To do so, we follow the argument of \cite{OgawaTsutsumi1}. Let us define a function
\[
\vartheta(r):= \left\{
\begin{array}{c l c}
2r & \text{if}& 0\leq r\leq 1, \\
2[r-(r-1)^3] &\text{if}& 1<r\leq 1+1/\sqrt{3}, \\
\vartheta' <0 &\text{if}& 1+ 1/\sqrt{3} <r < 2, \\
0 &\text{if}& r\geq 2,
\end{array}
\right.
\]
and 
\begin{align*}
\theta(r):= \int_{0}^{r}\vartheta(s)ds.
\end{align*}
It is easy to see that $\theta$ satisfies $(\ref{define theta})$. We thus define $\varphi_R$ as in $(\ref{define rescaled varphi})$. We show that $(\ref{positive condition})$ holds true for this choice of $\varphi_R$. Using the fact
\[
\Delta \varphi_R(x) = \varphi''_R(r) +\frac{d-1}{r}\varphi'_R(r),
\]
we have
\[
\chi_{2,R} = (2-b)(2-\varphi''_R) + (2d-2+b)\Big(2-\frac{\varphi'_R}{r}\Big).
\]
By the definition of $\varphi_R$, 
\[
\varphi'_R(r) = R\theta'(r/R) = R \vartheta(r/R), \quad \varphi''_R(r) = \theta''(r/R) = \vartheta'(r/R).
\]
\indent When $R<r \leq (1+1/\sqrt{3})R$, we have
\[
\chi_{1,R}(r) = 12 \Big(\frac{r}{R}-1\Big)^2, 
\]
and
\[
\chi_{2,R}(r) = 6\Big(\frac{r}{R}-1\Big)^2 \Big[2-b + \frac{(2d-2+b)(r/R-1)}{3r/R} \Big] <6\Big(\frac{r}{R}-1\Big)^2 \Big(2-b+ \frac{2d-2+b}{3\sqrt{3}}\Big).
\]
Since $0<r/R-1<1/\sqrt{3}$, we can choose $\ep>0$ small enough so that $(\ref{positive condition})$ is satisfied. \newline
\indent When $r> (1+1/\sqrt{3})R$, we see that $\vartheta'(r/R) \leq 0$, so $\chi_{1,R}(r) = 2(2-\varphi''_R(r)) \geq 4$. We also have that $\chi_{2,R}(r) \leq C$ for some constant $C>0$. Thus by choosing $\ep>0$ small enough, we have $(\ref{positive condition})$. 
\subsection{The case $d=1$ and $E(u_0)<0$}
We follow the argument of \cite{OgawaTsutsumi2}. We only consider the positive time, the negative one is treated similarly. We argue by contradiction and assume that the solution exists for all $t\geq 0$. We divide the proof in two steps.
\paragraph{\bf Step 1.} We assume that the initial data satisfies
\begin{align}
\delta := -16E(u_0) -C(1+N)^2\|u_0\|^{6-2b}_{L^2} -N\|u_0\|^2_{L^2} &>0, \label{condition 1 step 1} \\
\Big(\int \theta |u_0|^2 dx \Big)^{1/2} \Big( \frac{2}{\delta} \|\partial_x u_0\|^2_{L^2} +1\Big)^{1/2} & \leq \frac{1}{2} a_0, \label{condition 2 step 1}
\end{align}
where $C, N, \theta$ and $a_0$ are defined as in Lemma $\ref{lem virial estimate 1D}$. We will show that if $u_0$ satisfies $(\ref{condition 1 step 1})$ and $(\ref{condition 2 step 1})$, then the corresponding solution satisfies $(\ref{condition a_0})$ for all $t\geq 0$. Since $\theta(x)\geq 1$ for $|x|>1$ and $\delta>0$, we have from $(\ref{condition 2 step 1})$ that 
\begin{align}
\|u_0\|_{L^2(|x|>1)} \leq \frac{1}{2}a_0. \label{assumption initial data 1D}
\end{align}
Let us define
\[
T_0 :=\sup \{ t>0 : \|u(s)\|_{L^2(|x|>1)} \leq a_0, \quad 0\leq s <t \}.
\]
Since $s \mapsto \|u(s)\|_{L^2(|x|>1)}$ is continuous, $(\ref{assumption initial data 1D})$ implies $T_0>0$. If $T_0 =\infty$, we are done. Suppose that $T_0<\infty$. The continuity in $L^2$ of $u(t)$ gives
\begin{align}
\|u(T_0)\|_{L^2(|x|>1)} = a_0. \label{continuity T_0 1D}
\end{align}
On the other hand, $u(t)$ satisfies the assumption of Lemma $\ref{lem virial estimate 1D}$ on $[0,T_0)$. We thus get from Lemma $\ref{lem virial estimate 1D}$ and $(\ref{condition 1 step 1})$ that
\begin{align}
\int \theta |u(t)|^2 dx &\leq \int \theta |u_0|^2 dx - 2t \im{\int \partial_x \theta \overline{u}_0 \partial_x u_0 dx}  - \frac{\delta}{2} t^2 \nonumber \\
&= -\frac{\delta}{2} \Big(t+ \frac{1}{\delta} \im{ \int \partial_x \theta \overline{u}_0  \partial_x u_0 dx} \Big)^2 + \frac{1}{2\delta} \Big(\im{\int\partial_x \theta \overline{u}_0 \partial_x u_0 dx } \Big)^2 + \int \theta |u_0|^2 dx\nonumber \\
&\leq \frac{1}{2\delta} \Big(\im{\int\partial_x \theta \overline{u}_0 \partial_x u_0 dx } \Big)^2 + \int \theta |u_0|^2 dx \nonumber \\
&\leq \frac{1}{2\delta} \|\partial_x \theta u_0\|^2_{L^2} \|\partial_x u_0\|^2_{L^2} + \int \theta |u_0|^2 dx, \label{virial potential estimate 1D}
\end{align}
for all $0\leq t<T_0$. By the definition of $\theta$, it is easy to see that $\theta\geq \vartheta^2/4 = (\partial_x \theta)^2/4$ for any $x\in \R$. Thus, $(\ref{virial potential estimate 1D})$ yields
\[
\int \theta |u(t)|^2 dx \leq \Big(\frac{2}{\delta} \|\partial_x u_0\|^2_{L^2} +1 \Big) \int \theta |u_0|^2 dx,
\]
for all $0\leq t <T_0$. By $(\ref{condition 2 step 1})$ and the fact that $\theta \geq 1$ on $|x|>1$, we obtain
\[
\|u(t)\|_{L^2(|x|>1)} \leq \Big( \int \theta |u(t)|^2 dx\Big)^{1/2}  \leq \frac{1}{2}a_0,
\]
for all $0\leq t<T_0$. By the continuity of $u(t)$ in $L^2$, we get
\[
\|u(T_0)\|_{L^2(|x|>1)} \leq \frac{1}{2} a_0.
\]
This contradicts with $(\ref{continuity T_0 1D})$. Therefore, the assumptions of Lemma $\ref{lem virial estimate 1D}$ are satisfied with $I=[0,\infty)$ and we get
\[
\frac{d^2}{dt^2} V_\theta(t) \leq -\delta<0,
\]
for all $t\geq 0$. This is impossible. Hence, if the initial data $u_0$ satisfies $(\ref{condition 1 step 1})$ and $(\ref{condition 2 step 1})$, then the corresponding solution must blow up in finite time.
\paragraph{\bf Step 2.} In this step, we will use the scaling
\begin{align}
u_\lambda(t,x) = \lambda^{-\frac{1}{2}} u(\lambda^{-2}t, \lambda^{-1}x), \quad \lambda>0 \label{scaling mass-critical 1D}
\end{align}
to transform all initial data with negative energy into initial data satisfying $(\ref{condition 1 step 1})$ and $(\ref{condition 2 step 1})$. Note that the 1D mass-critical $\eqref{INLS}$ is invariant under $(\ref{scaling mass-critical 1D})$, that is, if $u(t)$ is a solution to the 1D mass-critical $\eqref{INLS}$ with initial data $u_0$, then $u_\lambda(t)$ is also a solution to the 1D mass-critical $\eqref{INLS}$ with initial data $u_\lambda(0)$. Moreover, we have
\begin{align}
\|u_\lambda(t)\|_{L^2} &= \|u_\lambda(0)\|_{L^2} = \|u_0\|_{L^2}, \label{mass scaling 1D} \\
E(u_\lambda(t)) &= E(u_\lambda(0)) = \lambda^{-2} E(u_0), \label{energy scaling 1D}
\end{align}
for any $t$ as long as the solution exists. \newline
\indent We will show that there exists $\lambda>0$ such that
\begin{align}
\delta_\lambda =-16E(u_\lambda(0)) -C (1+N)^2 \|u_\lambda(0)\|^{6-2b}_{L^2} - N\|u_\lambda(0)\|^2_{L^2} &>0, \label{condition scaling 1} \\
\Big(\int \theta |u_\lambda(0)|^2 dx \Big)^{1/2} \Big(\frac{2}{\delta_\lambda} \|\partial_x u_\lambda(0)\|^2_{L^2}+1 \Big)^{1/2} &\leq \frac{1}{2} a_0. \label{condition scaling 2}
\end{align}
By $(\ref{mass scaling 1D})$ and $(\ref{energy scaling 1D})$,
\begin{align}
\delta_\lambda = -16 \lambda^{-2} E(u_0) - C(1+N)^2 \|u_0\|^{6-2b}_{L^2} - N\|u_0\|^2_{L^2}. \label{delta lambda}
\end{align}
Thus, if we choose $\lambda>0$ so that
\begin{align}
\lambda <\Big[-16E(u_0) \Big[C(1+N)^2 \|u_0\|^{6-2b}_{L^2} + N\|u_0\|^2_{L^2} \Big]^{-1}\Big]^{1/2}=:\lambda_0,
\end{align}
then $(\ref{condition scaling 1})$ holds true. Moreover, since $\|\partial_x u_\lambda(0)\|^2_{L^2} = \lambda^{-2} \|\partial_x u_0\|^2_{L^2}$, we have from $(\ref{delta lambda})$ that
\begin{align}
\frac{2}{\delta_\lambda} \|\partial_x u_\lambda(0)\|^2_{L^2} = \frac{2\|\partial_x u_0\|^2_{L^2}}{\lambda^2 \delta_\lambda} \leq C_0, \quad 0<\lambda<\lambda_1, \label{condition scaling 2 proof}
\end{align}
for some $\lambda_1>0$, where $C_0$ depends on $\lambda_1$ but does not depend on $\lambda$. We next recall the following fact (see e.g. \cite[Lemma 2.3]{OgawaTsutsumi2}).
\begin{lem} \label{lem approximation}
Let $v\in L^2$ and 
\[
H(x):=\left\{
\begin{array}{cll}
|x|& \text{if}& |x|\leq 1, \\
1&\text{if}& |x|>1.
\end{array}
\right.
\]
Set $v_\lambda(x) = \lambda^{-1/2} v(\lambda^{-1} x)$ for $\lambda>0$. Then for any $\ep>0$, there exists $\lambda_0>0$ such that
\[
\|H v_\lambda\|_{L^2} \leq \ep, \quad 0<\lambda<\lambda_0.
\]
\end{lem} 
Applying Lemma $\ref{lem approximation}$, there exists $\lambda_2>0$  such that $\lambda_2<\lambda_1$ and 
\[
\int \theta |u_\lambda(0)|^2 dx \leq 4\|H u_\lambda(0)\|^2_{L^2} \leq \frac{1}{4}(C_0+1)^{-1} a_0^2, \quad 0<\lambda<\lambda_2.
\]
Combining this and $(\ref{condition scaling 2 proof})$, the condition $(\ref{condition scaling 2})$ holds for $0<\lambda<\lambda_2$. Therefore, if we choose $0<\lambda<\min\{\lambda_0, \lambda_2\}$, then $u_\lambda(0)$ satisfies $(\ref{condition scaling 1})$ and $(\ref{condition scaling 2})$. This completes the proof of the case $d=1$ and $E(u_0)<0$.

Combining three cases, we prove Theorem $\ref{theorem blowup mass-critical}$. \defendproof
\begin{rem} \label{rem sufficient condition}
We now show that the condition $E(u_0)<0$ is sufficient for the blowup but it is not necessary. Let $E>0$. We find data $u_0 \in H^1$ so that $E(u_0)=E$ and the corresponding solution $u$ blows up in finite time. We follow the standard argument (see e.g. \cite[Remark 6.5.8]{Cazenave}). Using the standard virial identity $(\ref{standard virial identity})$ with $\alpha=\alpha_\star$, we have
\[
\frac{d^2}{dt^2}\|x u(t)\|^2_{L^2} = 16 E(u_0),
\]
hence 
\[
\|x u(t)\|^2_{L^2} = 8 t^2 E(u_0) + 4t \Big(\im{\int \overline{u}_0 x \cdot \nabla u_0} dx\Big) + \|x u_0\|_{L^2}^2=:f(t).
\]  
We see that if $f(t)$ takes negative values, then the solution must blow up in finite time. In order to make $f(t)$ takes negative values, we need
\begin{align}
\Big(\im{\int \overline{u}_0 x \cdot \nabla u_0} dx \Big)^2 > 2E(u_0) \|x u_0\|^2_{L^2}. \label{sufficient and necessary blowup condition}
\end{align}
Now fix $\theta \in C^\infty_0(\R^d)$ a real-valued function and set $\psi(x)=e^{-i|x|^2} \theta(x)$. We see that $\psi \in C^\infty_0(\R^d)$ and 
\[
\im{\int \overline{\psi} x\cdot \nabla \psi dx} = -2 \int  |x|^2 \theta^2(x) dx <0. 
\]
We now set
\begin{align*}
\begin{aligned}
A &= \frac{1}{2} \|\nabla\psi\|^2_{L^2},& B &= \frac{1}{\alpha_\star +2} \int |x|^{-b}|\psi(x)|^{\alpha_\star+2}dx, \\
C &=\|x\psi\|^2_{L^2}, &  D &= -\im{\int \overline{\psi} x\cdot \nabla \psi dx}.
\end{aligned}
\end{align*}
Let $\lambda, \mu>0$ be chosen later and set $u_0(x) =\lambda \psi(\mu x)$. We will choose $\lambda, \mu>0$ so that $E(u_0) = E$ and $(\ref{sufficient and necessary blowup condition})$ holds true. A direct computation shows
\[
E(u_0) = \lambda^2 \mu^2 \mu^{-d} \frac{1}{2}\|\nabla\psi\|^2_{L^2} - \lambda^{\alpha_\star+2} \mu^b\mu^{-d} \frac{1}{\alpha_\star+2} \int |x|^{-b}|\psi(x)|^{\alpha_\star+2}dx = \lambda^2 \mu^{2-d} \Big(A-\frac{\lambda^{\alpha_\star}}{\mu^{2-b}} B\Big),
\]
and
\[
\im{\int \overline{u}_0 x \cdot\nabla u_0 dx} = \lambda^2 \mu^{-d} \im{\int \overline{\psi} x \cdot \nabla \psi dx} = -\lambda^2 \mu^{-d} D,
\]
and
\[
\|xu_0\|^2_{L^2} = \lambda^2 \mu^{-d-2} \|x\psi\|^2_{L^2} = \lambda^2 \mu^{-d-2} C.
\]
Thus, the conditions $E(u_0)=E$ and $(\ref{sufficient and necessary blowup condition})$ yield
\begin{align}
\lambda^2 \mu^{2-d} &\Big(A-\frac{\lambda^{\alpha_\star}}{\mu^{2-b}}B\Big) = E, \label{condition 1 sufficience}\\
\frac{D^2}{C}&> 2\Big(A-\frac{\lambda^{\alpha_\star}}{\mu^{2-b}}B\Big). \label{condition 2 sufficience}
\end{align}
Fix $0<\ep < \min\Big\{A, \frac{D^2}{2C} \Big\}$ and choose
\[
\frac{\lambda^{\alpha_\star}}{\mu^{2-b}} B =A-\ep.
\]
It is obvious that $(\ref{condition 2 sufficience})$ is satisfied. Condition $(\ref{condition 1 sufficience})$ implies
\[
\ep\lambda^2 \mu^{2-d} = E \quad \text{or} \quad \ep \Big(\frac{B}{A-\ep}\Big)^{\frac{2-d}{2-b}} \lambda^{2+\frac{(2-d)\alpha_\star}{2-b}} = E.
\]
This holds true by choosing a suitable value of $\lambda$. 
\end{rem}
\section{Intercritical case $\alpha_\star<\alpha<\alpha^\star$}
\setcounter{equation}{0}
In this section, we will give the proof of Theorem $\ref{theorem blowup intercritical}$. Let us consider separately two cases: $E(u_0)<0$ and $E(u_0) \geq 0$.
\subsection{The case $E(u_0)<0$}
\paragraph{\bf The case $xu_0 \in L^2$.}
By the standard virial identity $(\ref{standard virial identity})$ and the conservation of energy, we have
\begin{align*}
\frac{d^2}{dt^2} \|x u(t)\|^2_{L^2} &= 8 \|\nabla u(t)\|^2_{L^2} -\frac{4(d\alpha+2b)}{\alpha+2}\int |x|^{-b} |u(t,x)|^{\alpha+2} dx \\
&= 4(d\alpha+2b) E(u(t)) - 2(d\alpha - 4 +2b) \|\nabla u(t)\|^2_{L^2} < 4(d\alpha+2b) E(u_0) <0.
\end{align*}
Here $d\alpha-4+2b>0$ in the intercritical case $\alpha_\star<\alpha<\alpha^\star$. The standard convexity argument implies that the solution blows up in finite time.
\paragraph{\bf The case $u_0$ is radial.}
We use Lemma $\ref{lem localized virial identity}$ together with the conservation of energy to have for any $\ep>0$,
\begin{align*}
\frac{d^2}{dt^2}V_{\varphi_R} (t)&\leq 8\|\nabla u(t)\|^2_{L^2} - \frac{4(d\alpha+2b)}{\alpha+2} \int |x|^{-b}|u(t,x)|^{\alpha+2}dx \\
&\mathrel{\phantom{\leq}} + \left\{
\begin{array}{cl}
O\left( R^{-2} + R^{-[2(d-1) +b]} \|\nabla u(t)\|^2_{L^2} \right) &\text{if } \alpha =4 \\
O \left( R^{-2} + \ep^{-\frac{\alpha}{4-\alpha}} R^{-\frac{2[(d-1)\alpha+2b]}{4-\alpha}} + \ep \|\nabla u(t)\|^2_{L^2} \right) &\text{if } \alpha<4
\end{array}
\right.  \\
&= 4(d\alpha+2b) E(u_0) - 2(d\alpha-4 +2b) \|\nabla u(t)\|^2_{L^2}  \\
&\mathrel{\phantom{=}} + \left\{
\begin{array}{cl}
O\left( R^{-2} + R^{-[2(d-1) +b]} \|\nabla u(t)\|^2_{L^2} \right) &\text{if } \alpha =4, \\
O \left( R^{-2} + \ep^{-\frac{\alpha}{4-\alpha}} R^{-\frac{2[(d-1)\alpha+2b]}{4-\alpha}} + \ep \|\nabla u(t)\|^2_{L^2} \right) &\text{if } \alpha<4,
\end{array}
\right. 
\end{align*}
for any $t$ in the existence time. Since $d\alpha -4 +2b>0$, we take $R>1$ large enough when $\alpha=4$; and take $\ep>0$ small enough and $R>1$ large enough depending on $\ep$ when $0<\alpha<4$ to have that
\[
\frac{d^2}{dt^2}V_{\varphi_R} (t)\leq 2(d\alpha+2b) E(u_0)<0,
\]
for any $t$ in the existence time. This implies that the solution must blow up in finite time. 
\subsection{The case $E(u_0) \geq 0$}
In this case, we assume that the initial data $u_0$ satisfies $(\ref{condition below ground state})$ and $(\ref{condition blowup intercritical})$. We first show $(\ref{property blowup intercritical})$. 
By the definition of energy and multiplying both sides of $E(u(t))$ by $M(u(t))^\sigma$, the sharp Gagliardo-Nirenberg inequality $(\ref{sharp gagliardo-nirenberg inequality})$ yields
\begin{align}
E(u(t)) M(u(t))^\sigma &= \frac{1}{2} \Big( \|\nabla u(t)\|_{L^2} \|u(t)\|_{L^2}^\sigma\Big)^2 -\frac{1}{\alpha+2} \Big(\int |x|^{-b} |u(t,x)|^{\alpha+2} dx\Big) \|u(t)\|^{2\sigma}_{L^2} \nonumber \\
&\geq \frac{1}{2} \Big( \|\nabla u(t)\|_{L^2} \|u(t)\|_{L^2}^\sigma\Big)^2 -\frac{C_{\text{GN}}}{\alpha+2} \|u(t)\|^{\frac{4-2b-(d-2)\alpha}{2} + 2\sigma}_{L^2} \|\nabla u(t)\|^{\frac{d\alpha+2b}{2}}_{L^2} \nonumber \\
&=f(\|\nabla u(t)\|_{L^2} \|u(t)\|^\sigma_{L^2}), \label{define f}
\end{align}
where
\[
f(x) =\frac{1}{2} x^2 -\frac{C_{\text{GN}}}{\alpha+2} x^{\frac{d\alpha+2b}{2}}.
\]
Moreover, using $(\ref{pohozaev identities})$ and $(\ref{sharp constant gagliardo-nirenberg inequality})$, it is easy to see that
\begin{align}
f(\|\nabla Q\|_{L^2}\|Q\|^\sigma_{L^2}) = E(Q)M(Q)^\sigma. \label{property f}
\end{align}
We also have that $f$ is increasing on $(0,x_0)$ and decreasing on $(x_0, \infty)$, where
\[
x_0 = \Big[\frac{2(\alpha+2)}{(d\alpha+2b) C_{\text{GN}}}\Big]^{\frac{2}{d\alpha-(4-2b)}}.
\]
Using again $(\ref{pohozaev identities})$ and $(\ref{sharp constant gagliardo-nirenberg inequality})$, we see that $x_0$ is exactly $\|\nabla Q\|_{L^2} \|Q\|^\sigma_{L^2}$. By $(\ref{define f})$, the conservation of mass and energy together with the assumption $(\ref{condition below ground state})$ imply
\[
f(\|\nabla u(t)\|_{L^2} \|u(t)\|^\sigma_{L^2}) \leq E(u_0) M(u_0)^\sigma <E(Q) M(Q)^\sigma.
\]
Using this, $(\ref{property f})$ and the assumption $(\ref{condition blowup intercritical})$, the continuity argument shows 
\[
\|\nabla u(t)\|_{L^2} \|u(t)\|^\sigma_{L^2}> \|\nabla Q\|_{L^2}\|Q\|^\sigma_{L^2},
\] 
for any $t$ as long as the solution exists. This proves $(\ref{property blowup intercritical})$.\newline
\indent We next pick $\delta>0$ small enough so that 
\begin{align}
E(u_0) M(u_0)^\sigma \leq (1-\delta) E(Q)M(Q)^\sigma. \label{refined blowup estimate}
\end{align} 
This implies 
\begin{align}
f(\|\nabla u(t)\|_{L^2}\|u(t)\|^\sigma_{L^2}) \leq (1-\delta)E(Q)M(Q)^\sigma. \label{refined estimate f}
\end{align}
By Pohozaev identities $(\ref{pohozaev identities})$, we learn that
\begin{align}
E(Q)M(Q)^\sigma = \frac{d\alpha-(4-2b)}{2(d\alpha+2b)} (\|\nabla Q\|_{L^2} \|Q\|^\sigma_{L^2})^2. \label{relation 1}
\end{align}
Moreover, we have from the fact $x_0 = \|\nabla Q\|_{L^2} \|Q\|^\sigma_{L^2}$ that 
\begin{align}
C_{\text{GN}} = \frac{2(\alpha+2)}{d\alpha+2b} \frac{1}{(\|\nabla Q\|_{L^2}\|Q\|^\sigma_{L^2})^{\frac{d\alpha-(4-2b)}{2}}}. \label{relation 2}
\end{align}
By dividing both sides of $(\ref{refined estimate f})$ by $E(Q)M(Q)^\sigma$ and using $(\ref{relation 1})$ and $(\ref{relation 2})$, we obtain
\[
\frac{d\alpha+2b}{d\alpha-(4-2b)} \Big(\frac{\|\nabla u(t)\|_{L^2} \|u(t)\|^\sigma_{L^2}}{\|\nabla Q\|_{L^2} \|Q\|^\sigma_{L^2}} \Big)^2 - \frac{4}{d\alpha-(4-2b)} \Big(\frac{\|\nabla u(t)\|_{L^2} \|u(t)\|^\sigma_{L^2}}{\|\nabla Q\|_{L^2} \|Q\|^\sigma_{L^2}} \Big)^{\frac{d\alpha+2b}{2}} \leq 1-\delta.
\]
The continuity argument then implies that there exists $\delta'>0$ depending on $\delta$ so that
\begin{align}
\frac{\|\nabla u(t)\|_{L^2} \|u(t)\|^\sigma_{L^2}}{\|\nabla Q\|_{L^2} \|Q\|^\sigma_{L^2}} \geq 1+\delta' \quad \text{or} \quad \|\nabla u(t)\|_{L^2} \|u(t)\|^\sigma_{L^2} \geq (1+\delta')\|\nabla Q\|_{L^2} \|Q\|^\sigma_{L^2}. \label{refined property blowup solution}
\end{align}
\indent We also have that for $\ep>0$ small enough,
\begin{align}
8\|\nabla u(t)\|^2_{L^2} -\frac{4(d\alpha+2b)}{\alpha+2} \int |x|^{-b} |u(t,x)|^{\alpha+2} dx + \ep \|\nabla u(t)\|^2_{L^2} \leq -c<0, \label{refined virial estimate intercritical}
\end{align}
for any $t$ in the existence time. Indeed, multiplying the left hand side of $(\ref{refined virial estimate intercritical})$ with the conserved quantity $M(u(t))^\sigma$, we get
\[
\text{LHS}(\ref{refined virial estimate intercritical}) \times M(u(t))^\sigma = 4(d\alpha+2b) E(u(t)) M(u(t))^\sigma + (8+\ep-2d\alpha-4b) \|\nabla u(t)\|^2_{L^2} M(u(t))^\sigma.
\] 
The conservation of mass and energy, $(\ref{refined blowup estimate}), (\ref{relation 1})$ and $(\ref{refined property blowup solution})$ then yield
\begin{align*}
\text{LHS}(\ref{refined virial estimate intercritical}) \times M(u_0)^\sigma &\leq 4(d\alpha+2b) (1-\delta) E(Q)M(Q)^\sigma + (8+\ep -2d\alpha-4b) (1+\delta')^2 (\|\nabla Q\|_{L^2}\|Q\|^\sigma_{L^2})^2 \\
&=  (\|\nabla Q\|_{L^2}\|Q\|^\sigma_{L^2})^2 \Big[ 2(d\alpha-4+2b)[1-\delta - (1+\delta')^2] + \ep (1+\delta')^2 \Big].
\end{align*}
By taking $\ep>0$ small enough, we prove $(\ref{refined virial estimate intercritical})$. 
\paragraph{ \bf The case $xu_0 \in L^2$.} 
The finite time blowup for the intercritical $\eqref{INLS}$ with initial data in $H^1 \cap L^2(|x|^2dx)$ satisfying $(\ref{condition below ground state})$ and $(\ref{condition blowup intercritical})$ was proved in \cite{Farah}. For the sake of completeness, we recall some details. By the standard virial identity $(\ref{standard virial identity})$ and $(\ref{refined virial estimate intercritical})$,
\[
\frac{d^2}{dt^2} \|x u(t)\|^2_{L^2} = 8 \|\nabla u(t)\|^2_{L^2} -\frac{4(d\alpha+2b)}{\alpha+2}\int |x|^{-b} |u(t,x)|^{\alpha+2} dx \leq -c <0. 
\]
This shows that the solution blows up in finite time. 
\paragraph{\bf The case $u_0$ is radial.} 
We first note that under the assumptions of Theorem $\ref{theorem blowup intercritical}$, we can apply Lemma $\ref{lem localized virial identity}$ to obtain for any $\ep>0$,
\begin{align*}
\frac{d^2}{dt^2}V_{\varphi_R} (t)&\leq 8\|\nabla u(t)\|^2_{L^2} - \frac{4(d\alpha+2b)}{\alpha+2} \int |x|^{-b}|u(t,x)|^{\alpha+2}dx \\
&\mathrel{\phantom{\leq}} + \left\{
\begin{array}{cl}
O\left( R^{-2} + R^{-[2(d-1) +b]} \|\nabla u(t)\|^2_{L^2} \right) &\text{if } \alpha =4, \\
O \left( R^{-2} + \ep^{-\frac{\alpha}{4-\alpha}} R^{-\frac{2[(d-1)\alpha+2b]}{4-\alpha}} + \ep \|\nabla u(t)\|^2_{L^2} \right) &\text{if } \alpha<4,
\end{array}
\right. 
\end{align*}
for any $t$ in the existence time. Taking $R>1$ large enough when $\alpha=4$, and $\ep>0$ small enough and $R>1$ large enough depending on $\ep$ when $0<\alpha<4$, we learn from $(\ref{refined virial estimate intercritical})$ that
\[
\frac{d^2}{dt^2}V_{\varphi_R} (t)\leq -c/2<0.
\]
This shows that the solution must blow up in finite time. 

Combining two cases, we prove Theorem $\ref{theorem blowup intercritical}$.

\section*{Acknowledgments}
The author would like to thank his wife-Uyen Cong for her encouragement and support. He would like to thank his supervisor Prof. Jean-Marc Bouclet for the kind guidance and constant encouragement. He also would like to thank the reviewer for his/her helpful comments and suggestions, which greatly improved the presentation of the paper. 


\end{document}